%% file: Remarks_partitions_into_expanders.tex
\documentclass[11pt, a4paper]{amsart}
\usepackage[bookmarksopen=true]{hyperref}
\usepackage{amsfonts,amssymb,verbatim}
\usepackage{latexsym}
\usepackage{xspace}
\usepackage{enumerate}
\usepackage[usenames,dvipsnames]{xcolor}
\usepackage{amsthm}
\usepackage{amsmath}

\input{ExtraCommands.tex}

\newcommand{\ebdry}{\partial}
\newcommand{\fol}{F\o{}lner\xspace}

\title{Remarks on partitions into expanders}
\author{Federico Vigolo}

\begin{document}

 \begin{abstract}
  In this note we give a short proof that graphs having no linearly small \fol sets can be partitioned into a union of expanders. We use this fact to prove a partition result for graphs admitting linearly small maximal \fol sets and we deduce that a family of such graphs must contain a family of expanders. We also show that the existence of partitions into expanders is a quasi\=/isometry invariant.
 \end{abstract}
 
 \maketitle

\section{Introduction}
This paper revolves around the fact that ``graphs where linearly\=/small subsets have large boundaries can be partitioned into unions of linearly\=/large expanders''. To make the statement clear, we need to introduce some terminology: let $X$ be a finite graph with no multiple edges or loops.
Given a finite set of vertices $A\subseteq X$, the \emph{boundary} of $A$ is the set of edges connecting $A$ to its complement: 
\[
 \partial A=\braces{\{v,w\}\in \edge(X)\mid v\in A,\ w\in X\smallsetminus A}.
\]
Given $\epsilon>0$, a non\=/empty set of vertices $A\subset X$ is a \emph{$\epsilon$\=/\fol set} if $\abs{A}\leq\frac 12\abs{X}$ and $\abs{\partial A}\leq\epsilon\abs{A}$ (here $\abs{X}$ is the number of vertices in $X$). The graph $X$ is an \emph{$\epsilon$\=/expander} if it contains no $\epsilon$\=/\fol sets. Let $\deg(X)\coloneqq \max\braces{\deg(v)\mid v\in X}$ be the \emph{degree} of $X$ and $D\in \NN$ some number. Then $X$ is an $(\epsilon,D)$\=/expander if it is an $\epsilon$\=/expander and $\deg(X)\leq D$. 

 If $X$ is a connected finite graph, it is trivially a $(\frac{2}{\abs{X}},\abs{X})$\=/expander. On the other hand, it is generally hard and very interesting to prove that a graph $X$ is an $(\epsilon,D)$\=/expander for some constants $\epsilon$, $D$ that are fixed \emph{a priori} and do not depend on $\abs{X}$. A \emph{family of expander graphs} is a sequence of $(\epsilon, D)$\=/expanders $(X_n)_{n\in\NN}$ such that $\abs{X_n}\to\infty$. We refer to \cite{HLW06} for more background and motivation.

 A subset of vertices $Y\subset X$ can be made into a subgraph of $X$ by keeping all the edges in $X$ with both endpoints in $Y$ (such a graph is often called a \emph{full subgraph} of $X$). In this paper we will say that $X=X_1\sqcup\cdots\sqcup X_n$ is a \emph{partition} of $X$ if $X_i$ are subgraphs arising from a partition of the set of vertices of $X$. We do \emph{not} require that every edge of $X$ is an edge of $X_i$ for some $i$ (\emph{i.e.} there might be edges connecting the $X_i$'s). We will be particularly interested in partitions where the graphs $X_i$ are $\epsilon$\=/expanders. If $\deg(X)\leq D$, $X_i$ will then automatically be $(\epsilon,D)$\=/expanders.

 Finally, given a constant $\alpha\in (0,1)$ we say that a subset $A\subseteq X$ is \emph{$\alpha$\=/big} if $\abs{A}\geq\alpha\abs{X}$, and that it is \emph{$\alpha$\=/small} if $\abs{A}<\alpha\abs{X}$.
Given nested subsets $A\subseteq Y\subseteq X$, we will avoid confusion by specifying whether $A$ is $\alpha$\=/big \emph{in $Y$} or \emph{in $X$} (and similarly for $\alpha$\=/small).

 In this paper we wish to give an elementary proof of the fact that graphs with a ``small\=/set expansion on a linear scale'' can be partitioned into linearly\=/large expanders:
 
\begin{alphthm}
\label{thm:partition.no.small.folner}
Let $X$ be a finite graph. If $X$ has no $\alpha$\=/small $\epsilon$\=/\fol sets, then it can be partitioned as $X=X_1\sqcup \cdots \sqcup X_k$ where $k\leq\lfloor\frac{1}{\alpha}\rfloor$, all the $X_i$ are $\alpha$\=/big and they are $\delta$\=/expanders for $\delta=\frac \epsilon{4^k}$.
\end{alphthm}

 Apart from the specific constants, Theorem~\ref{thm:partition.no.small.folner} can be easily deduced from a---much more refined---result of Oveis Gharam--Trevisan \cite[Theorem 1.5]{gharan2014partitioning} (see Remark~\ref{rmk:previous.work} for a more detailed comparison). 
 The main contribution of this note is to provide a short self\=/contained proof for Theorem~\ref{thm:partition.no.small.folner} (Section~\ref{sec:appendix:strucure.theorem}). 
 
 We hope that this work will help popularize this basic but not\=/so\=/well\=/known fact. For this reason, we begin our exposition by illustrating a few geometric consequences of Theorem~\ref{thm:partition.no.small.folner} (Section~\ref{sec:related.results}). Namely, we prove a rather general partition result for graphs, and we also use Theorem~\ref{thm:partition.no.small.folner} to show that the property of admitting partitions into linearly large expanders is invariant under quasi\=/isometry. We find that Theorem~\ref{thm:general.partition}, Corollary~\ref{cor:no.folner.sets.have.expanders} and Theorem~\ref{thm:qi.invariance} should be of independent interest.
 
 It is also worth pointing out that the objects of interest of this note (expanders and \fol sets) are very much related to the notion of \emph{separation profile} introduced by Benjamini--Schramm--Tim{\'a}r \cite{BOT10}.
 Some analogues of the techniques explained below proved to be useful in that context as well \cite{coz2019separation,Hum17}.

 \subsection*{Acknowledgments} 
 I am very grateful to the anonymous referee for pointing out a mistake in the proof of Lemma~\ref{lem:cutting.cheeger.sets}. I would like to thank Emmanuel Breuillard for directing me to the paper \cite{gharan2014partitioning} and pointing out the argument explained in Remark~\ref{rmk:previous.work}. I am also thankful to Henry Bradford, Ana Khukhro, Kang Li and Jiawen Zhang for their helpful comments.
 
 This work was supported by the ISF Moked 713510 grant number 2919/19.

\section{Consequences of the main result}\label{sec:related.results}
\subsection{A general partition theorem}
 It is convenient to introduce a piece of notation: given subsets $A,B\subset X$, we let $\ebdry^BA$ be the set of edges in $X$ joining a vertex in $A$ with a vertex in $B\smallsetminus A$ (this is the subset of $\partial A$ consisting of edges that land in $B$).
 It is interesting to combine Theorem~\ref{thm:partition.no.small.folner} with the ``maximal \fol set trick'':
 
 \begin{lem}\label{lem:max.folner.trick}
  Let $X$ be a finite graph and $\epsilon>0$ a fixed constant. If there exists an $\epsilon$\=/\fol set $F$ that is maximal with respect to inclusion, consider the subgraph $Y\coloneqq X\smallsetminus F$. Then every subset $A\subset Y$ such that $\abs{A}\leq \frac 12\abs{X}-\abs{F}$ satisfies $\abs{\partial^YA}>\epsilon\abs{A}$.
 \end{lem}
 
 This sort of maximality argument is used fairly often in the theory of von Neumann algebras and it was also a key ingredient in \cite{khukhro2019structure}. The proof of Lemma \ref{lem:max.folner.trick} is completely elementary and can also be found in \cite[Lemma~3.1]{khukhro2019structure}. Together with Theorem~\ref{thm:partition.no.small.folner}, the maximality trick implies the following structure theorem:

 \begin{thm}\label{thm:general.partition}
  Let $X$ be a finite graph. If $X$ has a maximal $\epsilon$\=/\fol set $F$ that is $\alpha$\=/small for some $\alpha<\frac{1}{2}$, then there exists $\delta=\delta(\epsilon,\alpha)$ such that $X$ can be partitioned as 
  \[
  X=F\sqcup Y_1\sqcup\cdots \sqcup Y_k
  \]
  where the graphs $Y_i$ are $\delta$\=/expanders and are $(\frac{1}{2}-\alpha)$\=/big in $X$.
 \end{thm}
 \begin{proof}
  Let $F$ be an $\alpha$\=/small maximal $\epsilon$\=/\fol set and let $Y\coloneqq X\smallsetminus F$. If $A\subset Y$ is an $\epsilon$\=/\fol set of $Y$, then by Lemma~\ref{lem:max.folner.trick} we must have:
  \[
   \abs{A}>\frac 12\abs{X}- \abs{F}=\abs{Y}-\frac 12\abs{X}>\Bigparen{1-\frac{1}{2(1-\alpha)}}\abs{Y}=\frac 12 \Bigparen{\frac{1-2\alpha}{1-\alpha}}\abs{Y}.
  \]
  That is, $Y$ has no $\frac 12 \Bigparen{\frac{1-2\alpha}{1-\alpha}}$\=/small $\epsilon$\=/\fol sets. We can hence apply  Theorem~\ref{thm:partition.no.small.folner} to obtain a partition of $Y$ into $\delta$\=/expanders. 
 \end{proof}

\begin{cor}\label{cor:no.folner.sets.have.expanders}
 Let $(X_n)_{n\in\NN}$ be a sequence of finite graphs with $\deg(X_n)\leq D$ and $\abs{X_n}\to\infty$. Given a constant $\epsilon>0$, either there exist $\epsilon$\=/\fol sets $F_n\subset X_n$ such that $\limsup \frac{\abs{F_n}}{\abs{X_n}}=\frac{1}{2}$ or there exist $\alpha,\delta>0$ and $\alpha$\=/big subgraphs $Y_n\subseteq X_n$ that are $(\delta,D)$\=/expanders (these options are non\=/exclusive). 
\end{cor}

Note in particular that the graphs $Y_n$ in Corollary~\ref{cor:no.folner.sets.have.expanders} would be a family of expander graphs. A sample application of this result could be proving that some metric space $Y$ contains families of expanders: it may be possible to prove that $Y$ contains some graphs $X_n$ that do not have \fol sets of size $\approx \abs{X_n}/2$, and Corollary~\ref{cor:no.folner.sets.have.expanders} would then immediately imply that $Y$ contains some genuine expanders as well. This is relevant \emph{e.g.} in the study of coarse embeddings into Hilbert spaces \cite{Gro03,khukhro2019structure,Tes09}.

\subsection{Invariance under quasi-isometry }
Given $L,A>0$, a $(L,A)$\=/quasi-isometry is a function between metric spaces $f\colon (X,d_X)\to(Y,d_Y)$ such that
\[
 \frac 1L d_X(x,x')-A\leq d_Y(f(x),f(x'))\leq Ld_X(x,x')+A
\]
for every $x,x'\in X$, and such that for every $y\in Y$ there is an $x\in X$ with $d_Y(f(x),y)\leq A$. This notion is a cornerstone of geometric group theory \cite{Gro93}.

Connected graphs can be seen as metric spaces where the distance between two vertices is the length of the shortest path connecting them. 
It is well\=/known that quasi\=/isometries preserve expansion. More precisely, one can prove the following lemma (see the proof of \cite[Lemma 2.7.5]{Vig18Thes}\footnote{The current setting is somewhat different from that of \cite{Vig18Thes}. Most notably, \cite{Vig18Thes} is concerned with \emph{coarse equivalences} and \emph{vertex boundaries}. Yet, it is easy to adapt the argument  outlined there to our situation.
}):

\begin{lem}\label{lem:folner.is.preserved}
 For every $\epsilon>0$ there exists an $\eta=\eta(\epsilon,D,L,A)$ such that if $X$ and $Y$ are connected graphs with degree bounded by $D$ and $f\colon X\to Y$ is an $(L,A)$\=/quasi\=/isometry, then for every subset $F\subset Y$ with $\abs{\partial F}\leq \eta\abs{F}$ the preimage $T=f^{-1}(F)$ is non\=/empty and satisfies $\abs{\partial T}\leq \epsilon\abs{T}$.
\end{lem}

\begin{rmk}
 The value of $\eta$ degrades exponentially fast as a function of the quasi\=/isometry constants. The proof in \cite[Lemma 2.7.5]{Vig18Thes} works for $\eta=\epsilon D^{- p(L,A)}$ where $p$ is some polynomial.
\end{rmk}

If $f\colon X\to Y$ is an $(L,A)$\=/quasi\=/isometry between graphs, then any two vertices of $X$ that are at distance $L(A+2)$ or more are sent to distinct points in $Y$. It follows that if $X$ has degree bounded by $D$ then 
\[
\abs{f^{-1}(y)}\leq( \text{cardinality of a ball of radius $L(A+2)$}) \leq D^{L(A+2)+1}.
\] 
On the other hand, since every point in $Y$ is within distance $A$ from $f(X)$, it follows that if $Y$ has degree bounded by $D$ then 
\[
\abs{Y}\leq D^{A+1}\abs{f(X)}\leq D^{A+1}\abs{X}. 
\]
Combining these inequalities one can prove the following:

\begin{lem}\label{lem:small.is.preserved}
 Let $X$ and $Y$ be connected graphs with degree bounded by $D$ and $f\colon X\to Y$ an $(L,A)$\=/quasi\=/isometry. For any $\alpha>0$ let $\beta\coloneqq D^{-L(A+2)-A-2}\alpha$. Then the preimage of a $\beta$\=/small subset of $Y$ is $\alpha$\=/small in $X$.
\end{lem}

The following is now immediate:

\begin{prop}\label{prop:qi.invariance}
 For every $\epsilon,\alpha,D,L,A>0$ there exist $\eta,\beta>0$ such that if $X$ and $Y$ are connected graphs with degree bounded by $D$, $X$ has no $\alpha$\=/small $\epsilon$\=/\fol set and $f\colon X\to Y$ is an $(L,A)$\=/quasi\=/isometry, then $Y$ has no $\beta$\=/small $\eta$\=/\fol sets.
\end{prop}

We can use Proposition~\ref{prop:qi.invariance} and Theorem~\ref{thm:partition.no.small.folner} to show that the existence of partitions into linearly large expanders is invariant under quasi\=/isometry. More precisely, we can prove the following:

\begin{thm}\label{thm:qi.invariance}
 Fix $\epsilon,\alpha,D,L,A>0$ and $\eta,\beta>0$ as in Proposition~\ref{prop:qi.invariance}. Let $X_n$ and $Y_n$ be two sequences of finite graphs with degree bounded by $D$ and $f_n\colon X_n\to Y_n$ be $(L,A)$\=/quasi isometries. If each $X_n$ can be partitioned in $(2\alpha)$\=/large $\epsilon$\=/expanders then $Y_n$ can be partitioned into $\beta$\=/large $\delta$\=/expanders where $\delta=4^{-\lfloor1/\beta\rfloor}\eta$.
\end{thm}
\begin{proof}
 Since $X_n$ can be partitioned in $(2\alpha)$\=/large $\epsilon$\=/expanders, it does not contain any $\alpha$\=/small $\epsilon$\=/\fol set. Proposition~\ref{prop:qi.invariance} implies that $Y_n$ does not contain $\beta$\=/small $\eta$\=/\fol sets. The claim now follows from Theorem~\ref{thm:partition.no.small.folner}.
\end{proof}

\begin{rmk}
It would be fairly complicated to directly prove Theorem~\ref{thm:qi.invariance} by ignoring Theorem~\ref{thm:partition.no.small.folner} altogether. This is because quasi\=/isometries are not bijections and the techniques needed to prove Lemmata~\ref{lem:folner.is.preserved} and \ref{lem:small.is.preserved} are ill suited for constructing \emph{partitions} of the codomains.
\end{rmk}

\section{Proof of Theorem~\ref{thm:partition.no.small.folner}}
\label{sec:appendix:strucure.theorem}

Recall that, given subsets $A,B\subset X$, we denote by $\ebdry^BA$ the set of edges in $X$ joining a vertex in $A$ with a vertex in $B\smallsetminus A$. Define the \emph{Cheeger constant} of a finite graph as
\[
 h(X)\coloneqq \min\Bigbraces{\frac{\abs{\partial A}}{\abs A}\Bigmid A\subset X,\ 0<\abs{A}\leq\frac{1}{2}\abs{X}}.
\]
Note that $X$ is an $\epsilon$\=/expander if and only if $h(X)>\epsilon$.

\begin{lem}\label{lem:cutting.cheeger.sets}
 Given any $\epsilon>0$ and a graph $X$ with Cheeger constant $h\coloneqq h(X)\leq \frac\epsilon 2$, let $Y\subset X$ be a set with $0<\abs{Y}\leq \frac{1}{2}\abs{X}$ and $\frac{\abs{\ebdry Y}}{\abs{Y}}=h$. Then every $(\epsilon/4)$\=/\fol set of $Y$ is an $\epsilon$\=/\fol set of $X$.
 
 Furthermore, letting $Z\coloneqq X\smallsetminus Y$ we also have that every $(\epsilon/4)$\=/\fol set of $Z$ is an $\epsilon$\=/\fol set of $X$.
\end{lem}

\begin{proof}
 Let $A\subset Y$ be a set such that $\abs{\partial^XA}>\epsilon\abs{A}$, we need to show that it is \emph{not} an $(\epsilon/4)$\=/\fol set of $Y$. Note that $\ebdry^XA=\ebdry^YA\sqcup\ebdry^ZA$ and that $\ebdry^ZA\subset\ebdry^XY$. We have:
 \[
 \ebdry^X (Y\smallsetminus A)= \paren{\ebdry^X Y\smallsetminus \ebdry^Z A}
    \sqcup \paren{\ebdry^A (Y\smallsetminus A)},
 \]
 and hence
 \[
  \abs{\ebdry^X(Y\smallsetminus A)}
  = \abs{\ebdry^X Y}-\abs{\ebdry^Z A} +\abs{\ebdry^A (Y\smallsetminus A)}
    = \abs{\ebdry^X Y}-\abs{\ebdry^Z A}+ \abs{\ebdry^Y A}
 \]
 because $\ebdry^Y A=\ebdry^A (Y\smallsetminus A)$. By minimality, we thus obtain:
 \begin{equation}\label{eq.boundary.of.difference}
  h=\frac{\abs{\ebdry^X Y}}{\abs{Y}}\leq
  \frac{\abs{\ebdry^X (Y\smallsetminus A)}}{\abs{Y\smallsetminus A}}
    = \frac{\abs{\ebdry^X Y}-\abs{\ebdry^Z A}+ \abs{\ebdry^Y A}}{\abs{Y\smallsetminus A}}.
 \end{equation}

 For convenience, let $t\coloneqq \abs{A}/\abs{Y}$ and let $r\coloneqq \abs{\ebdry^Y A}/\abs{\ebdry^X A}$. 
 With the newly introduced notation, \eqref{eq.boundary.of.difference} becomes:
 \[
  h\leq \frac{\abs{\ebdry^X Y}-(1-r)\abs{\ebdry^X A}+ r\abs{\ebdry^X A}}{(1-t)\abs{Y}}
  =\frac{h}{1-t}+ \frac{2r-1}{(1-t)/t}\frac{\abs{\ebdry^X A}}{\abs A}.
 \]
 Rearranging the terms we obtain:
 \[
  -h\leq (2r-1)\frac{\abs{\ebdry^X A}}{\abs{A}}
 \]
and hence
 \begin{equation}\label{eq:r_inequality}
  r\geq \frac{1}{2}\Bigparen{1-h\frac{\abs A}{\abs{\ebdry^X A}} } > \frac{1}{4},
 \end{equation}
 where the last inequality follows from the assumptions $h\leq\frac \epsilon 2$ and $\frac{\abs {\partial^X A}}{\abs A}>\epsilon$. We can therefore conclude that $A$ is not an $(\epsilon/4)$\=/\fol set of $Y$ because
\[
 \abs{\ebdry^Y A}=r\abs{\ebdry^X A} >\frac 14\abs{\partial^X A} >\frac{\epsilon}{4}\abs{A}.
 \]

\

The proof of the ``furthermore'' part is similar. As above, let $A\subset Z$ be such that $\abs{\partial^X A}>\epsilon\abs{A}$. Now there are two possibilities. If $\abs{Z\smallsetminus A}\leq\frac{1}{2}\abs{X}$ then we have an analogue of \eqref{eq.boundary.of.difference}:
\[
  h
  \leq\frac{\abs{\ebdry^X (Z\smallsetminus A)}}{\abs{Z\smallsetminus A}}
    = \frac{\abs{\ebdry^X Z}-\abs{\ebdry^Y A}+ \abs{\ebdry^Z A}}{\abs{Z\smallsetminus A}}
 \]
 and the same argument implies that $\partial^ZA>\frac{\epsilon}{4}\abs{A}$.

 On the other hand, if $\abs{Z\smallsetminus A}>\frac{1}{2}\abs{X}$ then $\abs{Y\sqcup A}<\frac{1}{2}\abs{X}$, and therefore we have
 \[
  h
  \leq \frac{\abs{\ebdry^X (Y\sqcup A)}}{\abs{Y\sqcup A}}
    = \frac{\abs{\ebdry^X Y}-\abs{\ebdry^Y A}+ \abs{\ebdry^Z A}}{\abs{Y}+\abs{A}}
    \leq h +\frac{\abs{\ebdry^Z A}- \abs{\ebdry^Y A}}{\abs{Y}+\abs{A}},
 \]
 from which it follows that $\abs{\ebdry^Z A}\geq \abs{\ebdry^Y A}$ and hence 
 $\abs{\ebdry^ZA}>\frac \epsilon 2 \abs{A}$.
\end{proof}

\begin{proof}[Proof of Theorem~\ref{thm:partition.no.small.folner}]
Let $X$ be a finite graph with no $\alpha$\=/small $\epsilon$\=/\fol sets. We will show that it can be partitioned as $X=X_1\sqcup \cdots \sqcup X_k$ where $k\leq\lfloor\frac{1}{\alpha}\rfloor$, all the $X_i$ are $\alpha$\=/big and they are $\delta$\=/expanders for $ \delta\coloneqq \frac{\epsilon}{4^k}$.

 The idea is to apply Lemma~\ref{lem:cutting.cheeger.sets} recursively: if $X$ is not an $\frac\epsilon 2$\=/expander then $h(X)\leq \frac\epsilon 2$ and there exists a $Y_0\subset X$ that realizes the Cheeger constant. Since $X$ has no $\alpha$\=/small $\epsilon$\=/\fol sets, we deduce that $\abs{Y_0}\geq \alpha\abs{X}$. Letting $Y_1\coloneqq X\smallsetminus Y_0$, we have a partition $X=Y_0\sqcup Y_1$ where both $Y_i$ are $\alpha$\=/large. Importantly, it follows from Lemma~\ref{lem:cutting.cheeger.sets} that $\frac\epsilon 4$\=/\fol sets of $Y_i$ are also $\epsilon$\=/\fol sets of $X$.
 
 Let us now focus on $Y_0$: if it is an $\frac{\epsilon}{2\cdot 4}$\=/expander there is nothing to do. Otherwise, we can choose $Y_{00}\subset Y_0$ realizing the Cheeger constant. Such $Y_{00}$ is an $\frac \epsilon 4$\=/\fol set in $Y_0$ and hence an $\epsilon$\=/\fol set in $X$. It follows that $\abs{Y_{00}}\geq\alpha\abs{X}$. On the other hand, $Y_{01}\coloneqq Y_0\smallsetminus Y_{00}$ is at least as large as $Y_{00}$ and hence $\abs{Y_{01}}\geq\alpha\abs{X}$. Using Lemma~\ref{lem:cutting.cheeger.sets} we deduce that the partition $Y_0=Y_{00}\sqcup Y_{01}$ is such that every $\frac{\epsilon}{4^2}$\=/\fol set in $Y_{0i}$ is an $\frac\epsilon 4$\=/\fol set in $Y_{0}$ and hence an $\epsilon$\=/\fol set in $X$.

 One can thus continue to partition the sets $Y_{i_0i_1\cdots i_k}$ that appear using this procedure. This process ends because $X$ is a finite graph and all the subsets $Y_{i_0i_1\cdots i_k}$ obtained during this process are $\alpha$\=/big in $X$. In particular, when the process ends one has partitioned $X$ into at most $\lfloor\frac{1}{\alpha}\rfloor$ sets $X_1,\ldots,X_k$. Moreover, the worst possible expansion constant is what is obtained by the longest chain of consecutive applications of Lemma~\ref{lem:cutting.cheeger.sets}. This gives rise to the---rather generous---lower bound $\delta\geq \frac{\epsilon}{4^k}$.
\end{proof}

\begin{rmk}
 By examining the end of the proof of Theorem~\ref{thm:partition.no.small.folner}, we see that the natural lower bound on $\delta$ is actually $\frac{\epsilon}{2\cdot 4^{k-1}}$. With some extra care, it is further possible to improve this to $\frac{\epsilon}{4^{k-1}}$ (if the partition process goes all the way and produces $k$ components, then they must automatically be $\epsilon/{4^{k-1}}$ expanders as any $\epsilon/{4^{k-1}}$\=/\fol set should be $\alpha$\=/large in $X$ and hence take more than half of the graph). It is also possible to improve the base of the exponential, making it arbitrarily close to $\frac 12$ (instead of $\frac 14$). This is accomplished by modifying the statement of Lemma~\ref{lem:cutting.cheeger.sets}: Inequality~\eqref{eq:r_inequality} shows that tighter control on $h$ directly yields a better asymptotic behaviour. The optimal parameters for Lemma~\ref{lem:cutting.cheeger.sets} would then depend on the expected value on $k$. However, such a fine tuning seems unnecessary, as there are other arguments that yield polynomial estimates (see below).
\end{rmk}

\begin{rmk}\label{rmk:previous.work}
 As already remarked, Theorem~\ref{thm:partition.no.small.folner} follows easily from a result of Oveis Gharam--Trevisan. For every $m\geq 1$ one can define a higher order Cheeger constant $\rho_m(X)$ as
 \[
  \rho_m(X)\coloneqq \min\Bigbraces{ \max_{1\leq i\leq m} \frac{\abs{\partial A_i}}{\abs{A_i}}\Bigmid A_1,\ldots, A_m\subset X\text{ disjoint}}.
 \]
 \cite[Theorem 1.5]{gharan2014partitioning} implies that when $\rho_m(X)>0$  one can always find a partition $X=X_1\sqcup\cdots\sqcup X_l$ for some $l\leq m-1$ where the graphs $X_i$ are $\zeta$\=/expanders for some $\zeta=\zeta(m,\rho_m(X),\deg(X))$. To prove Theorem~\ref{thm:partition.no.small.folner} it is then enough to note that if there are no $\alpha$\=/small $\epsilon$\=/\fol sets in $X$ and $m=\lfloor\frac{1}{\alpha}\rfloor+1$, then for any choice of $m$ disjoint sets $A_1,\ldots,A_m$ at least one of them will be smaller than $\alpha\abs{X}$ and hence $\rho_m(X)>\epsilon$. It will hence be possible to partition $X$ into at most $\lfloor\frac 1\alpha\rfloor$\=/many $\zeta$\=/expanders.
 
 The proof of Oveis Gharam--Trevisan appears to be somewhat more involved than the proof we gave (it follows from \cite[Theorem 1.7]{gharan2014partitioning}), but it has a few significant advantages: it gives a bound on the number of edges connecting the sets in the partition, it applies to weighted graphs\footnote{%
 This is not an important difference: it is not hard to modify our proof to cover this case as well.} and it produces asymptotically better constants.
 
 With regard to constants: we wrote that the constant $\zeta$ of Oveis Gharam--Trevisan depends on the degree of $X$ because what they actually estimate is the \emph{conductance}\footnote{%
 This includes some normalization terms that take into account the degrees of vertices. It is natural to consider the conductance when one is planning to use the spectral characterization of expansion (as Oveis Gharam--Trevisan do). In this note we preferred the approach via Cheeger constants because it is marginally simpler to introduce.}.
 In particular, this makes it hard to compare directly the constants that we obtain. It appears that our approach provides sharper estimates when $k$ is very small (\emph{i.e.} for large $\alpha$). On the other hand, our estimate degrades exponentially fast with $k$, while that of Oveis Gharam--Trevisan degrades only quadratically. 
 
 One small advantage of our proof is that it is not immediately clear from the result of Oveis Gharam--Trevisan that all the sets $X_1,\ldots, X_l$ appearing in the partition are $\alpha$\=/big.
\end{rmk}

 \bibliographystyle{abbrv}
 \bibliography{MainBibliography}
\end{document}

%% file: ExtraCommands.tex
\usepackage[T1]{fontenc}	
\usepackage[utf8]{inputenc}			
\usepackage{mathtools} 					
\usepackage{microtype} 					
\usepackage{bbm}					
\usepackage{bm}						
\usepackage[shortcuts]{extdash} 			
\usepackage{subcaption}					
\usepackage{mparhack}					
\usepackage{accents}

\usepackage{todonotes}
		\newcommand{\NN}{\mathbb{N}}

\DeclarePairedDelimiter\abs{\lvert}{\rvert}		
\DeclarePairedDelimiter\paren{(}{)}			
\DeclarePairedDelimiter\braces{\{}{\}}			

	\newcommand{\Bigparen}[1]{\paren[\Big]{#1}}

	\newcommand{\Bigbraces}[1]{\braces[\Big]{#1}}
\newcommand{\Bigmid}{\mathrel{\Big|}}			

\DeclareMathOperator{\edge}{E}				

\theoremstyle{plain}
\newtheorem{thm}{Theorem}[section]				
\newtheorem{prop}[thm]{Proposition}		
\newtheorem{lem}[thm]{Lemma}						
\newtheorem{cor}[thm]{Corollary}

\newtheorem*{thm*}{Theorem}			\newtheorem*{theorem*}{Theorem}		
\newtheorem*{prop*}{Proposition}		\newtheorem*{proposition*}{Proposition}
\newtheorem*{lem*}{Lemma}			\newtheorem*{lemma*}{Lemma}			
\newtheorem*{cor*}{Corollary}			\newtheorem*{corollary*}{Corollary}
\newtheorem*{qu*}{Question}			\newtheorem*{question*}{Question}
\newtheorem*{conj*}{Conjecture}			\newtheorem*{conjecture*}{Question}
\newtheorem*{fact*}{Fact}
\newtheorem*{claim*}{Claim}

\newtheorem{alphthm}{Theorem}			

\theoremstyle{definition}

\newtheorem*{de*}{Definition}			\newtheorem{definition*}{Definition}	
\newtheorem*{notation*}{Notation}	
\newtheorem*{conv*}{Convention}			\newtheorem*{convention*}{Convention}

\theoremstyle{remark}
\newtheorem{rmk}[thm]{Remark}